\documentclass[12pt, a4paper]{amsart}
\usepackage{amscd,amssymb,eucal,eufrak,mathrsfs,amsmath}
\usepackage{hyperref} 
\usepackage{graphics,colortbl}

\input xy
\xyoption{all}
\usepackage{epsfig}

\newtheorem{thm}{Theorem}[section]
\newtheorem{cor}[thm]{Corollary}
\newtheorem{lem}[thm]{Lemma}
\newtheorem{prop}[thm]{Proposition}

\theoremstyle{definition}
\newtheorem{defin}[thm]{Definition}

\theoremstyle{remark}
\newtheorem{remark}[thm]{Remark}
\newtheorem{remarks}[thm]{Remarks}
\newtheorem{example}[thm]{Example}
\newtheorem{examples}[thm]{Examples}
\numberwithin{equation}{section}

\newcommand{\delete}[1]{} 
\newcommand{\nt}{\noindent}

\def\eps{{\varepsilon}}

\def\s{\sigma}

\def\eps{{\varepsilon}}

\newcommand{\sk}{\vskip 0.2cm}

\newcommand{\ben}{\begin{enumerate}}

\newcommand{\een}{\end{enumerate}}
\newcommand{\bit}{\begin{itemize}}

\newcommand{\eit}{\end{itemize}}

\def\R {{\mathbb R}}
\def\N {{\mathbb N}}
\def\Z {{\mathbb Z}}

\def\F{{\mathcal F}}

\def\diam{{\mathrm{diam}}}

\newcommand{\lan}{\langle}
\newcommand{\ran}{\rangle}

\def\QED{\nobreak\quad\ifmmode\roman{Q.E.D.}\else{\rm Q.E.D.}\fi}

\begin{document}

\title[]{Median pretrees and functions of bounded variation}

\author[]{Michael Megrelishvili}
\address{Department of Mathematics,
	Bar-Ilan University, 52900 Ramat-Gan, Israel}
\email{megereli@math.biu.ac.il}
\urladdr{http://www.math.biu.ac.il/$^\sim$megereli}

\date{September 16, 2020}



\thanks{This research was supported by a grant of the Israel Science Foundation (ISF 1194/19)} 

\begin{abstract}  
	We introduce functions of bounded variation on median algebras and study some properties for median pretrees. We show that if $X$ is a compact median pretree (e.g., a dendron) in its shadow topology then every function $f\colon  X \to \R$ of  bounded variation 
	has the point of continuity property (Baire 1, if $X$, in addition, is metrizable). We prove a generalized version of Helly's selection theorem for a sequence of functions 
	with total bounded variation defined on a Polish median pretree $X$.   
\end{abstract}

\maketitle 


\section{Introduction}

Our aim is to introduce functions of bounded variation on median algebras and pretrees (in particular, on dendrons).   
This was motivated by  recent papers \cite{Me-Helly,GM-D} and especially by a joint work with E. Glasner \cite[Remark 4.11]{GM-D}, where we deal with some applications of median pretrees  in topological dynamics. 

\sk
In the present work we prove the following theorems (\ref{t:PretBV} and \ref{t:GenHelly} below).  
\sk 
\nt \textbf{Theorem A.} 
\textit{Let $X$ be a median pretree (e.g., dendron or a linearly ordered space) such that its natural shadow topology is compact or Polish. Then every function $f\colon X \to \R$ 
  with bounded variation has the point of continuity property (Baire 1 class function, if $X$ is Polish).}

\sk 
	
\nt 	\textbf{Theorem B.}   (Generalized Helly's selection theorem) 
	\textit{Let $X$ be a Polish  (e.g., compact metrizable) median pretree. Then every sequence $\{f_n \colon X \to [c,d]\}_{n \in \N}$ of functions with total bounded variation $\leq r$ 
	has a pointwise converging subsequence, which converges to a function with variation $\leq r$.}

 \sk 
 Recall that a topological space $X$ is said to be \textit{Polish} if it is homeomorphic to a separable complete metric space. 
 A \textit{continuum} is a compact Hausdorff connected space. 
 A continuum $D$ is said to be a \textit{dendron} \cite{Mill-Wattel} if every pair of distinct points $u,v$ can be separated in $D$ by a third point $w$. 
 A metrizable dendron is called a \textit{dendrite}. 
 The class of dendrons is an important class of 1-dimensional treelike compact spaces, \cite{Mill-Wattel,Charatoniks}. Group actions on dendrites is an attractive direction in dynamical systems theory (see \cite{DuMo,GM-D} and references therein). 
 
 We define in Section 3 (Definitions \ref{d:BVnew} and \ref{d:LinBV}) functions of bounded variation on median algebras. 
 In Section \ref{s:MedPretr}, we recall definition and auxiliary properties of median pretrees. As to the \textit{point of continuity property} and \textit{fragmented functions}, see Subsection \ref{s:fr}. 
 Note that such functions play a major role in Bourgain-Fremlin-Talagrand theory, \cite{BFT} which in turn is strongly related to the classical work of Rosenthal \cite{Ros0}. One of the results from \cite{BFT} allows us to derive Theorem B from Theorem A. 
 
 Weaker versions of these theorems for linearly ordered spaces and BV functions were proved in \cite{Me-Helly} and for median pretrees and monotone functions in \cite{GM-D}. 
 
 
 \sk
 \section{Related structures} 
 \label{s:MedPretr} 
 
 \textit{Pretree} (in terms of B.H. Bowditch) is a useful treelike structure which naturally generalizes several important structures including linear orders and the betweenness relation on dendrons.

 \subsection{Pretrees} 
 
 \begin{defin} \label{d:B} 
  By a \textit{pretree} (see for example \cite{B, Mal14}), we mean a pair $(X,R)$, where 
 $X$ is a set and $R$ is a ternary relation on $X$
 (we write $\lan a,b,c \ran$ to denote $(a,b,c) \in R$) satisfying the following three axioms:
 \begin{itemize}
 	\item [(B1)] $\lan a,b,c \ran \Rightarrow \lan c,b,a \ran$. 
 	\item [(B2)] $\lan a,b,c \ran \wedge \lan a,c,b \ran \Leftrightarrow b=c$.
 	\item [(B3)] $\lan a,b,c \ran \Rightarrow \lan a,b,d \ran  \vee \lan d,b,c \ran$.
 \end{itemize} 
 In \cite{AN} such a ternary relation is called a \textit{B-relation}.  
 \end{defin}

 It is convenient to use also an interval approach. For every $u,v \in X$ define 
 $$
 [u,v]_X:=\{x \in X: \lan u,x,v \ran\}.
 $$ 
 Sometimes we write simply $[u,v]$, where $X$ is understood.

 \begin{remark} \label{r:PretreeInt} 
 	The conditions (A0),(A1),(A2),(A3), 
 	as a system of axioms, are equivalent to the above definition via (B1), (B2), (B3)
 	(see \cite{Mal14}).  
 	In every pretree $(X,R)$ for every $a,b,c \in X$, we have 
 	\begin{itemize}
 		\item [(A0)] $[a,b] \supseteq \{a,b\}$. 
 		
 		\item [(A1)] $[a,b]=[b,a]$. 
 		\item [(A2)] If $c \in [a,b]$ and $b \in [a,c]$ then $b=c$.  
 		
 		\item [(A3)]  $[a,b] \subseteq [a,c] \cup [c,b]$. 
 			
 	\end{itemize}  
 \end{remark}

\sk 
 
 Every subset $Y$ of $X$ carries the naturally defined betweenness relation. 
 In this case, the corresponding intervals are $[a,b]_Y=[a,b] \cap Y$ for every $a, b \in  Y$. 
 
 For every linear order $\leq$ on a set $X$, we have the induced pretree $(X,R_{\leq})$ 
 defined by 
 $$
 \lan a,b,c \ran \Leftrightarrow (a \leq b \leq c) \vee (c \leq b \leq a). 
 $$
 Note that the opposite linear order defines the same betweenness relation.

 A subset $A$ of a pretree $X$ is said to be \textit{convex} if $[a,b] \subset A$ for every $a,b \in A$. Intersection of convex subsets is convex (possibly empty). 
 For a subset $A \subset X$, the \textit{convex hull} $co(A)$ is the intersection of all convex subsets of $X$ which contain $A$.  
 
 
 \sk  
 Let us say that $a,b,c \in X$ are \textit{collinear} if 
 $$
 a \in [b,c] \vee b \in [a,c] \vee c \in [a,b].
 $$
 A subset $Y$ of $X$ is \textit{linear} (see \cite[Section 3]{Mal14}) if all $a,b,c \in Y$ are collinear. 
 
 By a \textit{direction} on a linear subset $Y$ in a pretree $X$, we
 mean a linear order $\leq$ on $Y$ such that, $R_{\leq}$ is just the given betweenness relation on $Y$.  Each nontrivial linear subset $Y$ in a pretree $X$ admits precisely two directions.  

 


 Following A.V. Malyutin \cite{Mal14} (which in turn follows to the terminology of 
 \newline P. de la Harpe and J.-P. Preaux), we define the so-called \textit{shadow topology}. Alternative names in related structures are: \textit{Lawson's topology} and \textit{observer's topology}. See the related  discussion in \cite{Mal14}. 
 
 Given an ordered pair $(u,v) \in X^2, u \neq v$, 
 let 
 $$
 S^v_u:=\{x \in X: u \in [x,v]\}
 $$
 be the \textit{shadow} in $X$ defined by the ordered pair $(u,v)$. 
 Pictorially, the shadow $S^v_u$ is cast by a point $u$ when the light source is located at the point $v$. 
 The family $\mathcal{S} = \{S^v_u: u,v \in X, u \neq v\}$ is a subbase for the closed sets
 of the topology $\tau_s$. 
 The complement of $S^v_u$ is said to be a \textit{branch}
 	$$\zeta_u^v:=X \setminus S^v_u=\{x \in X: u \notin [x,v]\}.$$
 The set of all branches $\{\zeta_u^v: u, v \in X, u \neq v\}$ is a subbase of the shadow topology. 
 
 \sk 
 In the case of a linearly ordered set, we get the \textit{interval topology}. 
 In general, for an abstract pretree, the shadow topology is 
 often 
 (but not always) Hausdorff. 
 Furthermore, by \cite[Theorem 7.3]{Mal14} a pretree equipped with its shadow topology is Hausdorff 
 if and only if, as a topological space, it can be embedded into a dendron. 
 
 
 \begin{lem} \label{l:PretrProp}  	Let $X$ be a pretree. 
 	\begin{enumerate}
 		\item \cite[Lemma 1.16 (A6,A7)]{Mal14} \ For every	$c \in [a,b]$ we have: 
 		\begin{enumerate}
 		\item 	$[a,c] \cap [c,b]=\{c\};$
 		\item $[a,c] \cup [c,b]=[a,b].$ 
 		\end{enumerate}

 		\item \label{l:ConvH}  \cite[Lemma 2.8]{Mal14}
 	For every subset $A \subset X$ its convex hull is 
 	$$co(A)=\cup \{[a,b]: a,b \in A\}.$$
 		\item  \cite[Lemma 3.3.4]{Mal14} 
 		$[a,b]$ is a convex linear subset for every $a,b \in X$. 
 		\item 	\label{l:LocConv} 
 		\cite[Lemma 5.10.2]{Mal14} Every branch is convex. Hence, every pretree is \textit{locally convex}. 
 		\item \cite[Prop. 6.5]{Mal14}  Let $S$ be a subset in a pretree $X$. Then the shadow topology on $S$ (regarded as a pretree with the structure induced by that of $X$) is contained in the relativization of the shadow topology on $X$ to $S$. If $S$ is convex in $X$, then the two topologies above coincide.
 		
 	\end{enumerate} 
 \end{lem}

 \sk
 \subsection{Median algebras and pretrees}

 A {\it median algebra} (see, for example, \cite{Vel, B}) is a pair $(X,m)$, 
 where the function $m\colon  X^3 \to X$ satisfies the following three axioms:
 \begin{itemize}
 	\item [(M1)]  $m(x,x,y)=x$.  
 	\item [(M2)]  $m(x,y,z)=m(y,x,z)=m(y,z,x)$.
 	\item [(M3)]  $m(m(x,y,z),u,v)=m(x,m(y,u,v),m(z,u,v)).$
 \end{itemize}
 
 This concept has been studied for a long time (Birkhoff-Kiss, Grau, Isbell) and has applications in abstract convex structures, \cite{Vel}. 
 
 Every distributive lattice $(L,\wedge,\vee)$  (e.g., any power set $P(S):=\{A: A \subset S\}$) is a median algebra with the median operation
 $$
 m(a,b,c):=(a \wedge b) \vee (b \wedge c) \vee (c \wedge a). 
 $$
 A very particular case of this is a linearly ordered set. 
 
 Let $(X,m)$ be a median algebra. A subset $Y \subseteq X$ is a \textit{subalgebra} if it is median-closed in $X$. 
 In a median algebra $(X,m)$ for every subset $A$, there exists the subalgebra $sp(A)$ generated by $A$.   
 This is the intersection of all subalgebras containing $A$.

 In every median algebra $(X,m)$, we have the naturally defined intervals 
 $$
 [a,b]:=\{m(a,x,b): x \in X\}.
 $$
This leads to the natural ternary relation $R_m$ defined by  
$\lan a,c,b \ran \ \text{iff} \ c=m(a,c,b),$ 
equivalently $\lan a,c,b \ran \ \text{iff} \ c \in [a,b]$. 
Note that not every median algebra is a pretree under the relation $R_m$. 
A subset $C$ of a median algebra is \textit{convex} if $[a,b] \subset C$ for every $a,b \in C$. 
Every convex subset is a subalgebra. 

For every triple $a,b,c$ in a pretree $X$ the {\it median} $m(a,b,c)$ is the intersection 
$$
m(a,b,c):=[a,b] \cap [a,c] \cap [b,c]. 
$$
When it is nonempty the median is a singleton, \cite{B,Mal14}. 
A pretree $(X,R)$ for which this intersection is always nonempty is called a \textit{median pretree}.

 \begin{remarks} \label{r:med} \ 
 	\ben 
 	\item Every median pretree $(X,R)$ is a \textit{median algebra}. The corresponding ternary relation $R_m$ induced by the median function coincides with $R$.

 	\item A map $f\colon  X_1 \to X_2$ between two median algebras is \textit{monotone} 
 	(i.e., $f[a,b] \subset [f(a),f(b)]$) if and only if $f$ is \textit{median-preserving} (\cite[page 120]{Vel}) if and only if $f$ is \textit{convex} (\cite[page 123]{Vel}) (convexity of $f$ means that the preimage of a convex subset is convex). 
 	
 	\item Every median pretree is Hausdorff (and normal) in its shadow topology (\cite[Theorem 7.3]{Mal14}). 
 	\item \cite[Prop. 6.7]{Mal14} In a median pretree, the convex hull of a closed 
 	set is closed. In particular, the intervals $[a,b]$ are closed subsets.
 	\item It is a well-known (nontrivial) fact that for every finite subset $F \subset X$ in a median algebra the induced subalgebra $sp(F)$ is \textit{finite}, \cite{Vel}. 
 	\een 
 \end{remarks}

 A \textit{compact (median) pretree} is a (median) pretree $(X,R)$ for which the shadow topology  $\tau_s$ is compact. \textit{Polish pretrees} can be defined similarly.

 \begin{examples} \label{ex:median} \ 
 	\ben 
 	\item 
 	Every dendron $D$ is a compact median pretree with respect to the standard betweenness relation $R_B$ ($w$ is \textit{between} $u$ and $v$ in $X$ if $w$ separates $u$ and $v$ or if $w \in \{u,v\}$). 
 	Its shadow topology is just the given compact Hausdorff topology on $D$ (see \cite{Mill-Wattel, Mal14}). 
 	
 	
 	\item Every linearly ordered set $(L,\leq)$ is a median pretree with respect to the median $m_{\leq}(a,b,c)=b$ iff $a \leq b \leq c$ or $c \leq b \leq a$. Its shadow topology is  the usual \textit{interval topology} of the order. We say that a subset $Y$ of a median algebra $(X,m)$ is a \textit{linear subset} if there exists a linear order $\leq$ on $Y$ such that the induced median function $m_{\leq}$ and the restriction of $m$ agree on $Y$.   
 	\item Let $X$ be a \textit{$\Z$-tree} (a median pretree with finite intervals $[u,v]$). 
 	Denote by $Ends(X)$  the set of all its \textit{ends}. 
 	According to \cite[Section 12]{Mal14} the set $X \cup Ends(X)$ carries a natural 
 	$\tau_s$-\textit{compact} median pretree structure.
 	\een
 \end{examples}

%
%

 \sk 
 \subsection{Fragmented functions} 
 \label{s:fr} 
 
 Recall the definition of fragmentability which comes from Banach space theory \cite{JR,N,JOPV} and effectively used also in dynamical systems theory \cite{Me-nz,GM1,GM-survey}. 
 We give only the case of functions into metric spaces. Lemma \ref{l:FinUnion} is true also where 
 the codomain is a uniform space. 
 
 \begin{defin}
 	Let $f\colon (X,\tau) \to (M,d)$ be a function from a topological space into a metric space. 
 	We say that $f$ is \textit{fragmented} if for every nonempty subset $A \subset X$ and every $\eps >0$ there exists a $\tau$-open subset $O \subset X$ such that $O \cap A$ is nonempty and $\diam(f(O \cap A)) < \eps$.  
 	If $M=\R$ then we use the notation  $f \in {\mathcal F}(X)$.
 \end{defin}
 
 

 \begin{lem} \label{l:fr} \  
 	\ben
 	\item \cite{GM1} 
 	When $X$ is compact or Polish, then $f\colon X \to \R$ is fragmented iff $f$ has the 
 	\emph{point of continuity property} (i.e., for every closed
 	nonempty $A \subset X$ the restriction $f|_{A}\colon A \to \R$ has a continuity point).
 	\item \cite[p. 137]{Dulst} For every Polish space $X$, we have $\F(X)=B_1(X)$, where $B_1(X)$ is the set of all Baire 1 functions $X \to \R$.  
 	\item \cite[Lemma 3.7]{Dulst} Let $X$ be a compact or a Polish space. Then the following conditions are equivalent for a function $f\colon X \to \R$. 
 	\begin{enumerate}
 		\item $f \notin \F(X)$; 
 		\item there exists a closed 
 		subspace $Y \subset X$ and real numbers $\alpha < \beta$ such that the subsets 
 		$f^{-1}(-\infty,\alpha) \cap Y$ and $f^{-1} (\beta,\infty) \cap Y$ are dense in $Y$. 
 			\end{enumerate}
 		\item \cite[Section 3]{BFT} For every Polish space $X$, every pointwise compact subset of $B_1(X)$ is sequentially compact (see also \cite[Thm 3.13]{Dulst}). 
 	\een
 \end{lem}
 
 \begin{lem} \label{l:FinUnion}  
 	Let $f\colon (X,\tau) \to (M,d)$ be a function from a topological space into a metric space. 
 	Suppose that $X=\bigcup_{i=1}^n Y_i$ is a finite covering of $X$ such that every $Y_i$ is closed in $X$ and every restriction function $f|_{Y_i}\colon (Y_i,\tau|_{Y_i}) \to (M,d)$ is fragmented. Then $f\colon (X,\tau) \to (M,d)$ is also fragmented.  
 	\end{lem}
 \begin{proof} 
 	Since finite union of closed subsets is closed one may reduce the proof to the case of two subsets. So, assume that $X=Y_1 \cup Y_2$ and $f|_{Y_1}\colon Y_1 \to M, f|_{Y_2}\colon Y_2 \to M$ are fragmented. Let $\eps >0$  and $A \subset X$ be a nonempty subset. We have to show that 
 	\begin{equation} \label{e:fr1} 
 	\exists O \in \tau \ \ \ O \cap A \neq \emptyset \ \text{and} \ \diam(f(O \cap A)) < \eps. 
 	\end{equation}

 	There are two cases:
 	(a) $A \subseteq Y_1 \cap Y_2$ and (b) $A \nsubseteq Y_1 \cap Y_2$. 
 	In the first case, using the fragmentability of $f|_{Y_1}$, choose $O \in \tau$ such that $(O \cap Y_1) \cap A \neq \emptyset$ and $\diam (f((O \cap Y_1) \cap A)) < \eps$. Since in case (a) we have $A \subset Y_1$, then $(O \cap Y_1) \cap A=O \cap A$. Hence, the condition \ref{e:fr1} is satisfied. 
 	\sk 
 	Now consider (b) $A \nsubseteq Y_1 \cap Y_2$. Then $(A \cap Y_1) \setminus Y_2 \neq \emptyset$ or $(A \cap Y_2) \setminus Y_1 \neq \emptyset$. We will check only the first possibility (the second is similar). 
 	Using the fragmentability of $f|_{Y_1}$, choose for the subset $(A \cap Y_1) \setminus Y_2 \subset Y_1$  an open subset $U \in \tau$ in $X$ such that 
 	$$(U \cap Y_1) \cap ((A \cap Y_1) \setminus Y_2) \neq \emptyset$$ 
 	and $\diam (f((U \cap Y_1) \cap (A \cap Y_1) \setminus Y_2)) ) < \eps$.
 	Now observe that 
 	$$(U \cap Y_1) \cap ((A \cap Y_1) \setminus Y_2)=(U \cap Y_1) \cap (A \cap Y_2^c)=(U \cap Y_2^c) \cap A.$$
 	Then $O:=U \cap Y_2^c$ is the desired open subset in $X$. 
 	\end{proof}

 \sk 
 \section{Functions of bounded variation} 
 \label{s:BV}

\subsection{Functions on linearly ordered sets} 

\begin{defin} \label{d:linBV}  \cite{Me-Helly} 
	Let $(X,\leq)$ be a linearly ordered set. 
	We say that a bounded function $f\colon  (X,\leq) \to \R$ has \textit{variation} $\Upsilon_{\leq}(f)$ not greater than $r$ 
	if 
	$$\sum_{i=1}^{n-1} |f(x_i)-f(x_{i+1})| \leq r$$ 
	for every choice of $x_1 \leq x_2 \leq \cdots \leq x_n$ in $X$.

\end{defin}


The following was proved in \cite{Me-Helly} using the particular case of order-preserving maps and Jordan type decomposition for functions with BV. 

\begin{thm} \label{t:BV} \cite{Me-Helly}
	 Let $(K,\leq)$ be a compact linearly ordered topological space (with its interval topology).  
	Every function $f\colon  K \to \R$ with bounded variation is fragmented. 
\end{thm}

\subsection{Functions on median algebras}

We examine two definitions (\ref{d:BVnew} and \ref{d:LinBV}) of BV for median algebras. Each of these definitions naturally generalize Definition \ref{d:linBV}.

	Let $(X,m)$ be a median algebra and $R$ be the induced betweenness relation (as in Remark \ref{r:med}.1), where, as before, we write $\lan a,x,b \ran$ instead of $(a,x,b) \in R$. 
	In particular, for dendrons it is exactly the standard betweenness relation. 
	 Recall that 
	$$
	\lan a,x,b \ran \Leftrightarrow x \in [a,b] \Leftrightarrow m(a,x,b)=x.
	$$
	
	Now, let $Y \subseteq X$ be a subset. A two-element subset (\textit{doublet}) $\{a,b\} \subset Y$ is said to be $Y$-\textit{adjacent} (or $Y$-\textit{gap}) if $\lan a,c,b \ran \Rightarrow c=a \ \text{or} \ c=b$ for every $c \in Y$. 
	In terms of intervals: 
	$[a,b]_X \cap Y =\{a,b\}$. 
	By $adj(Y)$ we denote the set of all $Y$-adjacent doublets. 

\sk

\begin{defin} \label{d:BVnew} 
	Let 
	$f\colon X \to \R$ be a bounded real valued function on a median algebra $(X,m)$ and $\s \subset X$ is a finite \textit{subalgebra}. By the \textit{variation} $\Upsilon (f,\s)$ of $\s$, we mean 
	
	\begin{equation} \label{nBV1}
	\Upsilon(f,\s) : =\sum_{\{a,b\} \in adj(\s)} |f(a)-f(b)|. 
	\end{equation}
	
	The least upper bound 
	$$
	\sup \{\Upsilon(f,\sigma): \ \s \ \text{is a finite subalgebra in} \ X\}
	$$
	is the  {\it variation} of $f$. Notation:  $\Upsilon(f)$. If it is bounded, say if  $\Upsilon(f) \leq r$ for a given positive $r \in \R$, then we write $f \in BV_r(X)$. If $f(X) \subset [c,d]$ for some $c \leq d$, then we write also $f \in BV_r(X,[c,d])$. One more notation: $BV(X):=\bigcup_{r>0} BV_r(X)$. 	 
	
\end{defin}

Note that $BV(X)$ is closed under linear operations. 

\sk 
Every linear subset in a median algebra is a subalgebra. So, Definition \ref{d:BVnew} naturally extends Definition \ref{d:linBV}. 
Another natural attempt for a generalization would be considering the sums $\Upsilon(f,\s)$ only for finite linear subsets $\s$  (and not for all finite subalgebras) as in the following definition.  

\begin{defin} \label{d:LinBV} In terms of Definition \ref{d:BVnew}, consider the least upper bound 
	$$
	\sup \{\Upsilon(f,\sigma): \ \s \ \text{is a finite linear subset in} \ X\}.
	$$
	Let us call it the  {\it linear variation} of $f$. Notation:  $\Upsilon^L(f)$. 
	Then $BV_r^L(X)$ and $BV^L(X)$ are understood like in Definition \ref{d:BVnew}. 	
\end{defin}

Since $\Upsilon^L(f) \leq \Upsilon(f)$ we get $BV(X) \subseteq BV^L(X)$. 
In general, this inclusion is proper for median pretrees. That is, $BV(X) \neq BV^L(X)$ (Example \ref{ex:BV}.3). 

 Every bounded monotone function $f\colon X \to \R$ on every median algebra $X$ belongs to $BV_r^L(X)$, with $r=\diam (f(X))$, because the restriction of $f$ on a linear subset with a direction is order preserving or order reversing. 
 In fact, even $f \in BV(X)$ 
 if $X$ is a median pretree (Corollary \ref{c:MONOT}.2). 
 It is not true, in general, for median algebras (Example \ref{ex:BV}.4). 

 Directly from the definitions, we have $\Upsilon(f|_Y) \leq \Upsilon(f)$ and $\Upsilon^L(f|_Y) \leq \Upsilon^L(f)$ for every median algebra $X$, its subalgebra $Y$ and a function 
 $f\colon X \to \R$.

\begin{remarks} \ 
	\begin{enumerate}
		\item In \cite{FJ} the authors study a treelike system -- ``rooted nonmetric tree". In paragraph 7.4 they define functions of bounded variation on such objects. This definition essentially differs from our definition. 
		\item 
		In this article, we examine Definition \ref{d:BVnew} mainly in the case when $X$ is a median pretree. 
		 Note that for functions on multidimensional objects (subsets of $\R^n$) there are several definitions for BV functions (see, for example, Vitali-Hardy-Krause type variation in  \cite{BEU,Leonov,Chist} and references therein). Such definitions and ideas probably would be useful also for abstract median algebras or for \textit{metric median spaces} with \textit{finite rank} in the sense of \cite{B-medmetr}.


	\end{enumerate}

\end{remarks}



Sometimes, we use the following relative version of Definition \ref{d:BVnew}. 

\begin{defin} \label{d:BVnewRel} 
	Let $S \subset X$ be a subset of a median algebra $X$ and $P(S)$ is the power set. 
 By an \textit{$S$-variation} $\Upsilon (f,\s)$ of $\s$ on $S$, we mean 
 \begin{equation} \label{nBV1Rel}
 \Upsilon(f,\s)|_S : =\sum_{\{a,b\} \in adj(\s) \cap P(S)} |f(a)-f(b)|. 
 \end{equation} 
 

The variation of f on $S \subset X$ can be defined similarly 
which we denote by $\Upsilon(f)|_S$.  
\end{defin}

\sk 
Clearly, $\Upsilon(f,\s)|_S \leq \Upsilon(f,\s)$ and $\Upsilon(f)|_S \leq \Upsilon(f)$ for every $S \subset X$. 

Let us say that the sets $A$ and $B$ are \textit{almost disjoint} if $A \cap B$ is at most a singleton. 

\sk 
\begin{lem} \label{l:ineq} 
	Let $\s$ be a finite subalgebra in a median algebra $X$. 
	\begin{enumerate}
		\item For every almost disjoint subsets $S_1, S_2$ in $X$, we have 
		$$
		\Upsilon(f,\s) \geq \Upsilon(f,\s)|_{S_1} + \Upsilon(f,\s)|_{S_2}.
		$$
			\item $\Upsilon(f,\s)|_{S} \leq \Upsilon(f,\s \cap S)$ for every subalgebra $S \subset X$. 
				\item $\Upsilon(f,\s)|_{C} = \Upsilon(f,\s \cap C)$ for every convex subset $C \subset X$. 
				\item $\Upsilon(f,\s) \geq  \Upsilon(f,\s \cap C_1) + \Upsilon(f,\s \cap C_2)$ for every almost disjoint convex subsets $C_1, C_2$ of $X$. 
	\end{enumerate}
\end{lem}
\begin{proof}
	(1) Trivial.
	
	(2) $\s \cap S$ is a finite subalgebra of $X$. Hence, $\Upsilon(f,\s \cap S)$ is well defined.  If $\{a,b\} \in adj(\s) \cap P(S)$, then $\{a,b\} \in adj(\s \cap S) $. 
	
	(3) By (2) it is enough to show the inequality $\Upsilon(f,\s)|_{C} \geq \Upsilon(f,\s \cap C)$. It suffices to prove that if $\{a,b\} \in adj(\s \cap C) $ then $\{a,b\} \in adj(\s)$. Assuming the contrary, let $\lan a,x,b \ran$ for some $x \in \s$ with $x \notin \{a,b\}$. Then $x \in [a,b] \setminus \{a,b\} \subset C$ by the convexity of $C$ and we get $\{a,b\} \notin adj(\s \cap C)$, a contradiction. 
	
	(4) Combine (1) and (3). 
	\end{proof} 

\begin{examples} \label{ex:BV} \  
	\ben 

	\item For a linearly ordered set $(X,\leq)$, consider the induced pretree with the median
	$$
	m(x,y,z)=y \Leftrightarrow x\leq y \leq z \vee z \leq y \leq x.
	$$
	Then $\Upsilon_{\leq} (f)=\Upsilon^L(f)=\Upsilon(f)$. So, in this case, Definitions \ref{d:linBV}, \ref{d:LinBV}, \ref{d:BVnew} agree. 
	
%
	\item Let $X=\{a,b,c,m\}$ be the ``4-element triod", where $m=m(a,b,c)$ is the only ``nontrivial median". Then for every $f\colon X \to \R$, we have 
	$$
	\Upsilon(f)=|f(a)-f(m)|+|f(b)-f(m)|+|f(c)-f(m)| 
	$$
	and $\Upsilon^L(f)$ is the maximum between three quantities: 
	
	$|f(a)-f(m)|+|f(b)-f(m)|$, \ \ \ $|f(b)-f(m)|+|f(c)-f(m)|$ 
	
	\nt and $|f(a)-f(m)|+|f(c)-f(m)|$. 
	So, in general, $\Upsilon^L(f) < \Upsilon(f).$
	\item 
	Let $X:=\{0, 1, \cdots\} = \{0\} \cup \N$. Define the following betweenness relation on $X$:  
	$$\lan x,0,y\ran \ \  \forall x \neq y \ \ \forall x,y \in \N \ \ \text{and} \  \ \lan x,x,y \ran, \ \lan x,y,y \ran \ \ \forall x, y \in X.$$ 
Then we get a pretree with the median 
	$$m\colon X^3 \to X, \ m(x,y,z)=0 \ \ \ \forall x\neq y\neq z \neq x,$$
	and $m(x,x,y)=m(y,x,x)=m(x,y,x)=x \ \ \forall x, y \in X$.
	The intervals are $[x,y]=[y,x]=\{x,0,y\} \ \ \forall x \neq y$ from $\N$, $[x,0]=[0,x] = \{x,0\}$ for every $x \in \N$ (and of course, $[x,x]=\{x\}$ \ $\forall x \in X$). 
	The corresponding shadow topology $\tau_s$ is the Alexandrov compactification of the discrete space $\N$ adjoining the limit point $0$. 
	
	\begin{itemize}
		\item [(a)] $BV(X) \neq BV^L(X)$. 
		\sk 
		Define the characteristic function of the singleton $\{0\}$  
		$$f \colon X \to \R, \ f(x)=0 \ \forall x \neq 0, f(0)=1.$$ 
		Then $\Upsilon(f,\s_n)=n$ for every subalgebra $\s_n=\{0,1, \cdots,n\}$. Hence, $\Upsilon(f)=\infty$. In contrast, the \textit{linear variation} is bounded, $\Upsilon^L(f)=2$.   
		
		\sk 
		\item [(b)] 
		The analog of Jordan's decomposition for the variations in Definitions \ref{d:BVnew}, \ref{d:LinBV} is not true for compact median pretrees. 
		\sk 
		
		Indeed, observe that monotone functions $\varphi \colon X \to \R$ have a very special form. 
		Namely, there exists a finite subset $F$ (with at most two elements) of $\N$ such that $\varphi(\N \setminus F)=\varphi(0)$. Now, define 
		$$f\colon X \to \R, \ \ f(0)=0, \ f(n)=\frac{1}{2^n} \ \ \forall n \in \N.$$
		Then $f \in BV(X) \subset BV^L(X)$ and $f$ is not a difference of any two monotone functions on $X$. 
	\end{itemize}

	\item Let $X=[0,1] \times [0,1]$ be the square with the $l_1$-metric $d_1$. Then $(X,d_1)$ is a \textit{metric median space}, \cite{B-medmetr}. It gives a median algebra $(X,m)$. 
	
	Monotone functions on this median algebra $(X,m)$ are, of course, in $BV^L$ but not necessarily in $BV$. Indeed, this happens, for example, for the characteristic function $f=\chi_{[\frac{1}{2},1] \times [0,1]}$ of the subset $[\frac{1}{2},1] \times [0,1]$ of $X$.  
%
%
%
	
	\een
\end{examples}

\sk 

\begin{prop} \label{p:factor}  
	Let $X$ and $Y$ be median pretrees, $f\colon Y \to \R$ be a bounded function and $h\colon X \to Y$ be a monotone map  
	\begin{equation*}
	\xymatrix { X \ar[dr]_{f \circ h} \ar[r]^{h} & Y
		\ar[d]^{f} \\
		& \R }
	\end{equation*}
	Suppose that $\s_1$ is a finite subalgebra in $X$ 
	and $\s_2$ is a finite subalgebra in $Y$ such that 
	$h(\s_1) \subset \s_2$. 
	Then we have 
	$$
	\Upsilon (f \circ h,\s_1) \leq   \Upsilon (f,\s_2) \ \ \ \text{and} \ \ \  \Upsilon(f \circ h) \leq  \Upsilon(f).
	$$
	
\end{prop}
\begin{proof} 
	It is enough to show $\Upsilon (f \circ h,\s_1) \leq  \Upsilon (f,\s_2)$. 

	Let $\{s,t\} \in adj(\s_1)$. Consider the interval $[h(s),h(t)]_{\s_2}$ which is finite (because $\s_2$ is finite). 
	 By Lemma \ref{l:PretrProp}.3 it is a linear subset. Let 
	 $$
	 [h(s),h(t)]_{\s_2}=\{h(s)=y_1, y_2, \cdots, y_{n-1}, y_n=h(t)\} 
	 $$
	be its list of distinct elements linearly ordered according to the direction where $h(s)$ is the smallest element. 
	It is possible that $\{h(s),h(t)\} \notin adj(\s_2)$ (i.e., $n>2$). 
	 
	  For every $i<j<k$ we have $\lan y_i,y_j,y_k \ran$. 
Say that the doublet $\{y_i,y_{i+1}\}$ (from $Y$) is \textit{$\{s,t\}$-linking}, where $1 \leq i \leq n-1$. 
Using Lemma \ref{l:PretrProp}.2, every 
$\{s,t\}$-linking 
doublet $\{y_i,y_{i+1}\}$ (where $1 \leq i \leq n-1$) is $\s_2$-adjacent. 
	Clearly, 
	$$|(f \circ h)(s)-(f \circ h)(t)|=|f(h(s))-f(h(t))| \leq \sum_{i=1}^{n-1} |f(y_i)-f(y_{i+1})|.$$ 
	Now, in order to check $\Upsilon (f \circ h,\s_1) \leq  \Upsilon (f,\s_2)$, it is enough to verify that the $h$-images of two $\s_1$-adjacent doublets cannot contain common linking doublets. 
	For this it is enough to prove the following 
	
		\sk 
\nt 	\textbf{Claim:}  
	If $\{s_1,t_1\} \in adj(\s_1)$ and $\{s_2,t_2\} \in adj(\s_1)$ then $[h(s_1),h(t_1)]_Y$ and  $[h(s_2),h(t_2)]_Y$ are almost disjoint.  
	\sk 
	\begin{proof}
	First of all note that the subset $S:=\{s_1,t_1,s_2,t_2\} \subset X$ is linear (in particular, a subalgebra of $\s_1$). 
	Indeed, $m(s_1,t_1,s_2) \in \{s_1,t_1\}$.  Otherwise, $\{s_1,t_1\}$ is not adjacent in the subalgebra $\s_1$.
	This implies that $s_1 \in [s_2,t_1] \vee t_1 \in [s_1,s_2]$. Therefore, $s_1,t_1,s_2$ are collinear in $X$. 
	 Similarly, for any other triple from $S$. 
	 Choose one of the two possible compatible directions (linear orders) $\leq$ on $S$. 
		 
The function $h\colon X \to Y$ is monotone means that $h$ preserves the betweenness relation. Equivalently, $h([x,y]) \subset [h(x),h(y)]$. Therefore, $h$ preserves the collinearity of every triple in $S$. It follows that $h(S)$ is also a linear subpretree (in $Y$). 
Fix a linear order $\preccurlyeq$ on $h(S)$ which induces the linear betweenness. 

Without loss of generality, we can suppose that $s_1< t_1 \leq s_2 <t_2$ in $S$. Then $h(s_1) \preccurlyeq h(t_1) \preccurlyeq h(s_2) \preccurlyeq h(t_2)$ or
 $h(t_2) \preccurlyeq h(s_2) \preccurlyeq h(t_1) \preccurlyeq  h(s_1)$. Otherwise, $h$ is not monotone. 
 We provide the verification only for the first case because the second case is similar. So, let
 \begin{equation} \label{eq:inc0} 
h(s_1) \preccurlyeq h(t_1) \preccurlyeq h(s_2) \preccurlyeq h(t_2).
 \end{equation} 
 In order to prove the \textbf{Claim} (completing the proof of Proposition \ref{p:factor}), it is enough to check 
 \begin{equation} \label{eq:intersection} 
 [h(s_1),h(t_1)]_Y \cap [h(s_2),h(t_2)]_Y \subseteq \{h(t_1)\} \cap \{h(s_2)\}.
 \end{equation} 
%
The inclusion \ref{eq:intersection} is true by the following arguments. First of all, Equation (\ref{eq:inc0}) guarantees that $h(s_2) \in [h(t_1),h(t_2)]$. Lemma \ref{l:PretrProp}.1(b) implies that
\begin{equation} \label{eq:inc1} 
[h(s_2),h(t_2)]_Y \subseteq [h(t_1),h(t_2)]_Y.
\end{equation}
Since $ h(s_1) \preccurlyeq h(t_1) \preccurlyeq h(t_2)$, we have $h(t_1) \in [h(s_1),h(t_2)]_Y$. 
By Lemma \ref{l:PretrProp}.1(a), we obtain  
\begin{equation} \label{eq:inc2} 
[h(s_1),h(t_1)]_Y \cap [h(t_1),h(t_2)]_Y=\{h(t_1)\}. 
\end{equation}

Combining Equations (\ref{eq:inc1}) and (\ref{eq:inc2}), we have 
$$
[h(s_1),h(t_1)]_Y \cap [h(s_2),h(t_2)]_Y \subseteq [h(s_1),h(t_1)]_Y \cap [h(t_1),h(t_2)]_Y=\{h(t_1)\}.  
$$
Similarly, by Lemma \ref{l:PretrProp}.1 and Equation (\ref{eq:inc0}), we obtain  
$[h(s_1),h(t_1)]_Y \subseteq [h(s_1),h(s_2)]_Y$ and 
$[h(s_1),h(s_2)]_Y \cap [h(s_2),h(t_2)]_Y=\{h(s_2)\}.$ This implies 
\begin{equation} \label{eq:inc3}  
[h(s_1),h(t_1)]_Y \cap [h(s_2),h(t_2)]_Y \subseteq [h(s_1),h(s_2)]_Y \cap [h(s_2),h(t_2)]_Y=\{h(s_2)\}.
\end{equation}
Finally, Equations (\ref{eq:inc2}) and (\ref{eq:inc3}) establish (\ref{eq:intersection}).   
		\end{proof}    \end{proof}

\begin{cor}  \label{c:MONOT} 
	Let $X$ be a median pretree.  
\begin{enumerate} 
		\item 
For every pair of finite subalgebras $\s_1,\s_2$ of $X$ with $\s_1 \subseteq \s_2$ and every bounded function $f\colon X \to \R$, we have 
$
\Upsilon (f,\s_1)  \leq \Upsilon(f,\s_2).  
$
\item 
Let $h\colon X \to [c,d] \subset \R$ be a monotone  bounded map on $X$. Then $h \in BV_r(X)$, where $r=d-c$. 
\end{enumerate}
\end{cor}
\begin{proof}
	 Apply Proposition \ref{p:factor} for:  
	
	(1) the identity map $h=id_X$ and $f\colon X \to \R$. 
	
	(2) the map $h\colon X \to [c,d]$ and the inclusion map $f\colon [c,d] \hookrightarrow \R$. 
	\end{proof}

\sk 
\begin{example}
	If we allow in Definition \ref{d:BVnew} that the subset $\s_1$ of $X$ is not necessarily a subalgebra, then 
	the ``monotonicity law" $\Upsilon (f,\s_1) \leq \Upsilon(f,\s_2)$ is not true in general. 
	For example, take the 4-element triod $X=\{a,b,c,m\}$ (Example \ref{ex:BV}.2) and define the function 
	$$f\colon X \to [-1,1], \  f(a)=f(c)=f(m)=1, f(b)=0.$$
	Then for the subset $\s_1=\{a,b,c\}$ (which is not a subalgebra) and $\s_2=X$, we have $\Upsilon (f,\s_1)=2$ but $\Upsilon(f,\s_2)=1$. 
\end{example}

If $X$ is a median algebra and not necessarily median pretree, then the set $M(X)$ of all monotone 
maps $X \to \R$  is not necessarily a subset of $BV(X)$, as we see by Example~\ref{ex:BV}.4. 
  

\sk 

\begin{prop} \label{p:sum} 
	Let $C_1,C_2$ be convex almost disjoint subsets in a median pretree $X$. 
	 For every bounded function $f\colon X \to \R$, denote by $f|_{C_1}\colon C_1 \to \R$ and $f|_{C_2} \colon C_2 \to \R$ the restrictions. Then we have 
	$$
	\Upsilon (f) \geq \Upsilon (f|_{C_1})+\Upsilon (f|_{C_2}). 
	$$ 
\end{prop}
\begin{proof} 
	Let $\s_1, \s_2$ be finite subalgebras in $X$ such that $\s_1 \subset C_1$, $\s_2 \subset C_2$.
	It is enough to show that there exists a finite subalgebra $\s^*$ in $X$ such that 
	$$\Upsilon(f,\s^*) \geq \Upsilon(f,\s_1) + \Upsilon(f,\s_2).$$
	
	 Consider the subalgebra $\s^*:=sp(\s_1 \cup \s_2)$ of $X$ which is \textit{finite} by Remark \ref{r:med}.5.  Then $\s_1^*:=\s^* \cap C_1$ and $\s_2^*:=\s^* \cap C_2$ are finite subalgebras in $C_1$ and $C_2$, respectively. 
	Clearly, $\s_1 \subset \s_1^*, \s_2 \subset \s_2^*$. By Lemma \ref{l:ineq} we have 
	$$
	\Upsilon(f,\s^*) \geq  \Upsilon(f,\s^* \cap C_1) + \Upsilon(f,\s^* \cap C_1)= 
	\Upsilon(f,\s_1^*) + \Upsilon(f,\s_2^*). 
	$$
	Proposition \ref{p:factor} guarantees that $\Upsilon(f,\s_1^*) \geq \Upsilon(f,\s_1), \Upsilon(f,\s_2^*) \geq \Upsilon(f,\s_2)$. So we get 
	$\Upsilon(f,\s^*) \geq \Upsilon(f,\s_1) + \Upsilon(f,\s_2)$, as desired. 
	\end{proof}


\begin{thm} \label{t:PretBV}  
	Let $X$ be a median pretree (e.g., dendron or a linearly ordered space) such that its 
	 shadow topology is compact or Polish. Then every function $f\colon X \to \R$  
	with bounded variation has the point of continuity property. It is equivalent to say that $f$ is fragmented (Baire 1 class function, if $X$ is Polish).  

\end{thm}
\begin{proof}
	Let $f\colon X \to \R$ not satisfy the point of continuity property. That is, $f$ is not fragmented (Lemma \ref{l:fr}.1). Then 
	by Lemma \ref{l:fr}.3 
	there exists a closed (necessarily infinite) subspace $Y \subset X$ and real numbers $\alpha < \beta$ such that 
	\begin{equation} \label{e:dense} 
	cl(f^{-1}(-\infty,\alpha) \cap Y)=	cl(f^{-1} (\beta, \infty) \cap Y)=Y. 
	\end{equation}
	
	Assuming the contrary let $f\colon X \to \R$ have BV. By Definition \ref{d:BVnew}, there exists $r \in \R$ such that 
	$$
	\Upsilon(f)\:=\sup \{\Upsilon(f,\sigma): \ \s \ \text{is a finite subalgebra in} \ X \}=r. 
	$$ 
	Choose a finite subalgebra $\sigma_1 \subset X$ such that 
	$$r - \Upsilon(f,\sigma_1) < \beta -\alpha,$$
	where $$\Upsilon(f,\s_1)=\sum_{\{a,b\} \in adj(\s_1)} |f(a)-f(b)|.$$
	By Lemma \ref{l:PretrProp}.2, $co(\s_1)=\cup \{[c_i,c_j]: c_i, c_j \in \s_1\}$. 
		Since $\s_1$ is finite, by Remarks \ref{r:med}.4, its convex hull $co(\s_1)$ is closed (hence also compact (or, respectively,  Polish) in the subspace topology) in $X$. 
	We have to check two cases.

	\sk
\nt	\textbf{Case 1:} $Y \subseteq co(\s_1)$.  
	\sk 
	
	In this case, by Lemma \ref{l:fr}.3 (for the compact (or, Polish) space $co(\s_1)$), we obtain that the restriction map $f|_{co(\s_1)} \colon co(\s_1) \to \R$ is \textit{not fragmented}. 
	
	By Corollary \ref{c:MONOT}, the variation of the restricted map  
	$\Upsilon (f|_{co(\s_1)}) \leq \Upsilon(f) \leq r$ is also bounded. On the convex subset $co(\s_1) \subset X$, the (median) pretree structure induces exactly the subspace topology by Lemma \ref{l:PretrProp}.5. 
	
	Every interval $[c_i,c_j]$ has a linear order by Lemma \ref{l:PretrProp}.3 such that two variations defined above are the same (Example \ref{ex:BV}.1).  
	By Theorem \ref{t:BV} every restriction $f|_{[c_i,c_j]}$ has BV. Each of the intervals $[c_i,c_j]$ is closed in the shadow topology (Remark \ref{r:med}.4)). Therefore, by Lemma \ref{l:FinUnion} we obtain that $f|_{co(\s_1)} \colon co(\s_1) \to \R$ is \textit{fragmented}.
	This contradiction shows that Case 1 is impossible. 

	\sk
\nt	\textbf{Case 2:} $Y \nsubseteq co(\s_1)$. 
	\sk

	Choose a point $y_0 \in Y$ such that $y_0 \notin co(\s_1)$. 	
Recall that $co(\s_1)$ is closed 
in $X$. 
 Every pretree is locally convex by Lemma \ref{l:PretrProp}.4.  
		Therefore, there exists an open neighborhood $O$ of $y_0$ in $X$ such that $O$ is convex (one may choose it as a finite intersection of branches) in $X$ and $O \cap  co(\s_1) = \emptyset$.

	Choose $u,v \in O$ such that $u \in f^{-1}(-\infty,\alpha) \cap Y$ and $v \in  
	f^{-1} (\beta, \infty) \cap Y$. Since $O$ is convex, we have $[u,v] \subset O$. 
Then $[u,v] \cap  co(\s_1) = \emptyset$. Since $[u,v]$ and $co(\s_1)$ are disjoint convex subsets in $X$, we can apply Proposition \ref{p:sum} which yields 
	$$
	\Upsilon(f)  \geq \Upsilon (f,\s_1) + |f(u)-f(v)|.  
	$$
	By our choice of $\s_1$ and $r$,  
	it follows that $r < \Upsilon(f)=r$. 
	This contradiction completes the proof.
\end{proof}

\sk

\begin{prop} \label{p:closed} Let $X$ be a median pretree. Then 
$BV_r(X,[c,d])$ is pointwise closed and hence a compact subset in $[c,d]^X$. 
\end{prop}
\begin{proof}
	Let $\{f_i\}_{i \in I}$ be a net of functions in $BV_r(X,[c,d])$ such that $f \colon X \to [c,d]$ is its pointwise limit. For every finite subalgebra $\s$ of $X$ and every $i \in I$, we have 
	$$
	\Upsilon(f_i,\s) : =\sum_{\{a,b\} \in adj(\s)} |f_i(a)-f_i(b)| \leq r.
	$$
	   Since $f$ is the pointwise limit of  $\{f_i\}_{i \in I}$, we get 
	   $\lim |f_i(a)-f_i(b)|=\lim |f(a)-f(b)|$ for every given $\{a,b\} \in adj(\s)$. This implies that $\Upsilon(f,\s) \leq r$ for every finite subalgebra $\s$. Hence, $\Upsilon(f) \leq r$.
	\end{proof}

\sk 
\subsection{Generalized Helly's selection principle} 

Note that there exists a sequence of functions $\{f_n: [0,1] \to [0,1]\}_{n \in \N}$ without any pointwise convergence subsequence. Indeed, the compact space $[0,1]^{[0,1]}$ (and even $\{0,1\}^{[0,1]}$) is not sequentially compact. 

Recall the following classical result of Helly, \cite{Helly, Natanson}. 

\sk 
\sk 
\nt 	{\bf Helly's Selection Theorem}:  
\textit{For every sequence of functions $\{f_n \colon [a,b] \to [c,d]\}_{n \in \N}$ with total variation $\leq r$, there exists a pointwise convergent subsequence.} 
\sk 

This result remains true replacing $[a,b]$ by any abstract linearly ordered set as it was proved in \cite{Me-Helly}. 
Our Theorem \ref{t:PretBV} allows us to prove 
the following generalization.

\begin{thm} \label{GenHellyThm}\label{t:GenHelly}  (Generalized Helly's selection theorem) 
	Let $X$ be a Polish 
	median pretree (e.g., dendrite) and $\{f_n \colon X \to [c,d]\}_{n \in \N}$ be a sequence of real functions which has total bounded variation $\leq r$. 
	 Then there exists a pointwise converging subsequence which converges to a function with variation $\leq r$.  That is, $BV_r(X,[c,d])$ is sequentially compact. 
\end{thm}
\begin{proof}
	By Theorem \ref{t:PretBV} the set $BV_r(X,[c,d])$ is a subset of $\F(X)$. Since $X$ is Polish we have  $\F(X)=B_1(X)$ (Lemma \ref{l:fr}.2). At the same time,  $BV_r(X,[c,d])$ is compact (by Proposition \ref{p:closed}). It is well known that by the  Bourgain--Fremlin--Talagrand theorem (Lemma \ref{l:fr}.4) for every Polish $X$ every pointwise compact subset of $B_1(X)$ is sequentially compact. Hence,   $BV_r(X,[c,d])$ is sequentially compact. 
\end{proof}

\sk 

\begin{remark}
	There are many natural BV functions on dendrites which are not monotone. For example, consider the real triod $X=[u,v] \cup [v,w] \cup [u,w]  \subset \R^2$, where $[u,v] \cap [v,w] \cap [u,w]=\{m\}$.
	Every ``coloring" $f\colon X \to \{1,2,3\}$, provided that every ``open arc" $(x,y)$ is monochromatic, is a function with BV. 
	Much more generally, $f$ is with BV if and only if every of three restrictions on the corresponding intervals are BV functions. 
	However, many such functions are not monotone. For example, if we use all three colors and if $f(m) \neq 2$, then $f$  is not monotone.  
\end{remark}

\bibliographystyle{amsplain}

\begin{thebibliography}{10}


\bibitem{AN} 
S.A. Adeleke and P.M. Neumann, \textit{Relations related to betweenness: their structure and automorphisms,}
Mem. Amer. Math. Soc. \textbf{131} (1998), n. 623. 





\bibitem{BFT} J. Bourgain, D.H. Fremlin and M. Talagrand,
{\it Pointwise compact sets in Baire-measurable functions}, Amer.
J. of Math.,  \textbf{100} (1977), n. 4, 845--886.

\bibitem{B}
B.H. Bowditch, \textit{Treelike structures arising from continua and convergence groups}, 
Mem. Am. Math. Soc., \textbf{139} (1999), n. 662. 


\bibitem{B-medmetr} 
B.H. Bowditch, \textit{Some properties of median metric spaces,} 
Groups, Geometry, and Dynamics, \textbf{10} (2016). 

\bibitem{BEU} 
H.D. Brunk, G.M. Ewing and W.R. Utz, 
\textit{Some Helly theorems for monotone functions},  
Proc. of AMS, 
 \textbf{7} (1956), No. 5, 776--783.  

\bibitem{Charatoniks} 
J.J. Charatonik and W.J. Charatonik, \textit{Dendrites}, 
Aportaciones Matematica, Serie Comunicaciones \textbf{22} (1998), 227--253. 



\bibitem{Chist} 
V.V. Chistyakov and U.V. Tretyachenko, \textit{Maps of several variables of finite total variation. II. \newline E. Helly-type pointwise selection principles}, 
J. Math. Anal. Appl. \textbf{369} (2010), 82--93.


\bibitem{Dulst}  
D. van Dulst, \emph{Characterizations of Banach spaces not containing $l^1$}, 
Centrum voor Wiskunde en Informatica, Amsterdam, 1989.



\bibitem{DuMo}
B. Duchesne and N. Monod, \textit{Group actions on dendrites and curves}, Ann. Inst. Fourier, Grenoble \textbf{68} (2018), 2277--2309.

\bibitem{FJ} 
C. Favre and M. Jonnson, \textit{The valuative tree}, Lecture Notes in Math., \textbf{1853}, Springer, 2004. 



%



%


\bibitem{GM1}
E. Glasner and M. Megrelishvili, \emph{Linear representations of
	hereditarily non-sensitive dynamical systems}, Colloq. Math.,
\textbf{104} (2006), n. 2, 223--283.

\bibitem{GM-survey}
E. Glasner and M. Megrelishvili,
\emph{Representations of dynamical systems on Banach spaces,}
in: Recent Progress in General Topology III, 
(Eds.: K.P. Hart, J. van Mill, P. Simon),  
Springer-Verlag, Atlantis Press, 2014, 399--470. 

\bibitem{GM-D}
E. Glasner and M. Megrelishvili, 
\emph{Group actions on treelike compact spaces}, Science China Math., 
 \textbf{62} (2019), n. 12, 2447--2462.  

\bibitem{Helly}
E. Helly, {\it Uber lineare Funktionaloperationen}, Sitzungsber. Naturwiss. Kl. Kaiserlichen Akad. Wiss. Wien \textbf{121} (1912), 265--297. 

\bibitem{JOPV}
J.E. Jayne, J. Orihuela, A.J. Pallares and G. Vera, {\it
	$\sigma$-fragmentability of multivalued maps and selection
	theorems}, J. Funct. Anal. {\bf 117} (1993), n. 2, 243--273.

\bibitem{JR}
J.E. Jayne and C.A. Rogers, {\it Borel selectors for upper
	semicontinuous set-valued maps}, Acta Math. {\bf 155} (1985),
41--79.




\bibitem{Leonov} 
A.S. Leonov, \textit{On the total variation for functions of several variables and a multidimensional analog of Helly’s selection principle}, Math. Notes \textbf{63} (1998), n. 1, 61--71.

\bibitem{Mal14} 
A.V. Malyutin, \textit{Pretrees and the shadow topology}, 
St. Petersburg Math. J., \textbf{26} (2015), 225--271.


\bibitem{Me-nz}
M. Megrelishvili, {\em Fragmentability and representations of
	flows\/}, Topology Proceedings, {\bfseries 27} (2003), n. 2, 497--544. 
See also: arXiv:math/0411112. 

\bibitem{Me-Helly}
M. Megrelishvili, \emph{A note on tameness of families having bounded variation}, 
Topology Appl. \textbf{217} (2017), 20--30. 


%

\bibitem{N}
I. Namioka, {\em Radon-Nikod\'ym compact spaces and
	fragmentability\/}, Mathematika {\bfseries 34} (1987), 258-281.

\bibitem{Natanson} 
I.P. Natanson, {\it Theory of functions of real variable}, v. I, Frederick Ungar, New York, 1964. 

\bibitem{Ros0}
H.P. Rosenthal,
{\em A characterization of Banach spaces containing $\ell_1$\/},
Proc.\ Nat.\ Acad.\ Sci.\ (USA)
{\bfseries 71} (1974), 2411--2413.


\bibitem{Vel}
M.L.J. van de Vel, \textit{Theory of convex structures}, North-Holland Math. Library, \textbf{50} (1993). 



\bibitem{Mill-Wattel}
J. van Mill and E. Wattel, \textit{Dendrons,} in: Topology and Order structures, edited by: H.R. Bennett, D.J. Lutzer, Part 1, Math. Centre Tracts \textbf{142}, Math. Centrum, Amsterdam 1981, pp. 59--81. 

\end{thebibliography}

\end{document}